\documentclass{amsart}

\usepackage{amssymb,amsmath,comment,mathtools,tikz-cd,todonotes}

\usepackage[colorlinks=true,linkcolor=blue]{hyperref}

\theoremstyle{remark}

\newtheorem{para}{\bf}[section]

\newtheorem{example}[para]{\bf Example}

\newtheorem{rem}[para]{\bf Remark}

\theoremstyle{definition}

\newtheorem{dfn}[para]{Definition}

\theoremstyle{plain}

\newtheorem{theorem}[para]{Theorem}

\newtheorem{lemma}[para]{Lemma}
\newtheorem{cor}[para]{Corollary}
\newtheorem{prop}[para]{Proposition}

\newenvironment{numequation}{\addtocounter{para}{1}
\begin{equation}}{\end{equation}}

\newcommand{\bbB}{{\mathbb B}}
\newcommand{\bbC}{{\mathbb C}}

\newcommand{\bbG}{{\mathbb G}}
\newcommand{\bbH}{{\mathbb H}}

\newcommand{\bbN}{{\mathbb N}}

\newcommand{\bbQ}{{\mathbb Q}}
\newcommand{\bbR}{{\mathbb R}}

\newcommand{\bbZ}{{\mathbb Z}}

\newcommand{\frg}{{\mathfrak g}}
\newcommand{\frh}{{\mathfrak h}}

\newcommand{\frH}{{\mathfrak H}}

\newcommand{\cA}{{\mathcal A}}

\newcommand{\cC}{{\mathcal C}}
\newcommand{\cD}{{\mathcal D}}

\newcommand{\cF}{{\mathcal F}}
\newcommand{\cG}{{\mathcal G}}

\newcommand{\cO}{{\mathcal O}}

\newcommand{\DHnn}{D(\bbH_n^\circ, H)}

\newcommand{\DlHN}{{D^{\rm la}(H,E)}}
\newcommand{\Danh}{D^{an}(\bbH, E)}

\newcommand{\Danhn}{D^{an}(\bbH_n^{\circ}, E)}
\newcommand{\ClaZp}{C^\la(\Zp, E)}
\newcommand{\an}{{\rm an}}

\newcommand{\car}{\stackrel{\simeq}{\longrightarrow}}
\newcommand{\cm}{{\mathfrak m}_{\bbC_p}}

\newcommand{\cont}{{\rm cont}}

\newcommand{\der}{\partial}

\newcommand{\eps}{{\epsilon}}

\newcommand{\Hom}{{\rm Hom}}
\newcommand{\hot}{{\widehat{\otimes}}}
\newcommand{\hra}{\hookrightarrow}
\newcommand{\id}{{\rm id}}
\newcommand{\im}{{\rm im}}

\newcommand{\la}{{\rm la}}

\newcommand{\Lie}{{\rm Lie}}

\newcommand{\lra}{\longrightarrow}
\newcommand{\midc}{{\; | \;}}
\newcommand{\Mod}{{\rm Mod}}
\newcommand{\Ne}{{\mathbb N}}

\newcommand{\ot}{\otimes}

\newcommand{\pr}{{\rm pr}}

\newcommand{\Q}{{\mathbb Q}}

\newcommand{\Qp}{{\bbQ_p}}
\newcommand{\ra}{\rightarrow}

\newcommand{\Sp}{{\rm Sp}}

\newcommand{\Spf}{{\rm Spf}}

\newcommand{\sub}{\subset}

\newcommand{\Sym}{{\rm Sym}}

\newcommand{\vpi}{{\varpi}}

\newcommand{\Z}{{\mathbb Z}}

\newcommand{\Zp}{{\mathbb Z_p}}

\begin{document}

\title{Rigid Analytic Vectors in Locally Analytic Representations}



\author{}
\address{}
\curraddr{}
\email{lahiria@iu.edu}
\thanks{}

\author{Aranya Lahiri}
\address{Department of Mathematics, Indiana University, Bloomington}
\curraddr{}
\email{}
\thanks{}


\date{}

\dedicatory{} 

\begin{abstract}
Let $H$ be a uniform pro-$p$ group. Associated to $H$ are rigid analytic affinoid groups $\bbH_n$, and their ``wide open" subgroups $\bbH_n^{\circ}$. Denote by $D^\la(H)= C^\la(H)'_b$ the locally analytic distribution algebra of $H$ and by $\DHnn$ Emerton's ring of $\bbH_n^{\circ}$-rigid analytic distributions on $H$. If $V$ is an admissible locally analytic representation of $H$, and if $V_{\bbH_n^\circ-\an}$ denotes the subspace of $\bbH_n^\circ$-rigid analytic vectors (with its intrinsic topology), then we show that the continuous dual of  $V_{\bbH_n^\circ-\an}$ is canonically isomorphic to $\DHnn \ot_{D^\la(H)} V'$. From this we deduce the exactness of the functor $V \rightsquigarrow V_{\bbH_n^\circ-\an}$ on the category of admissible locally analytic representations of $H$.
\end{abstract}

\numberwithin{equation}{section}

\maketitle
\tableofcontents

\section{Introduction}
Let $\bbQ_p\subseteq F\subseteq E$ be finite extensions 
of $\bbQ_p$. Locally analytic representations of 
locally $F$-analytic groups as introduced by 
Schneider-Teitelbaum \cite{S-T} have occupied a central 
stage in the $p$-adic local Langlands program in recent 
years. For a compact $F$-analytic group $G$ the locally 
analytic distribution algebra $D^\la(G,E)$ associated 
to $G$ carries the structure of a Fr\'echet-Stein 
algebra, and the category of coadmissible modules 
over $D^\la(G,E)$ is, by definition, anti-equivalent to the 
category of admissible locally $F$-analytic  
representations of $G$ over $E$, cf. \cite{S-T2}. 

\vskip8pt

To describe the main results of this paper, we assume in the introduction that the base field $F$ is $\bbQ_p$. 
Throughout this paper we denote by $H$ a uniform 
pro-$p$ group. There is canonically associated to $H$ a  rigid analytic affinoid group $\bbH$ and a sequence of affinoid subgroups $\bbH_n \sub \bbH$, for $n \in \Ne$. 
Each $\bbH_n$ in turn contains a canonically defined ``wide open" subgroup $\bbH_n^\circ$ (which is in general not affinoid). Let $\DlHN:= C^\la(H,E)'_b$ be the locally analytic distribution algebra, which is the strong dual of the space of $E$-valued locally analytic functions on $H$. The multiplication on $\DlHN$ is a convolution product. Furthermore, denote by $C^\la(H,E)_{\bbH_n^{\circ}-\an} \sub C^\la(H,E)$ the subspace of functions which are rigid analytic for the translation action of $\bbH_n^\circ$, and let $\DHnn:= (C^\la(H,E)_{\bbH_n^{\circ}-\an})'_b$ be its strong dual space, which is again equipped with the structure of an $E$-algebra. The inclusion $C^\la(H,E)_{\bbH_n^\circ-\an} \hra C^\la(H,E)$ induces by duality a morphism of $E$-algebras 

\begin{numequation}\label{can_morph} \DlHN \lra \DHnn \;.
\end{numequation}

If $V$ is an admissible locally analytic representation of $H$ over $E$, then, by definition, the continuous dual space $V'_b$ (equipped with the topology of bounded convergence) is a coadmissible $\DlHN$-module. The space of $\bbH_n^\circ$-rigid analytic vectors $V_{\bbH_n^\circ-\an}$ of $V$ (which carries an intrinsic topology finer than the subspace topology) plays a crucial role in M. Emerton's treatment of locally analytic representation theory, cf. \cite{Em17}. In that \cite[Thm. 1.2.11, Thm. 6.1.19]{Em17} Emerton has shown that the canonical homomorphism of topological $\DHnn$-modules

\begin{numequation}\label{can}
\DHnn \ot_\DlHN V'_b \lra (V_{\bbH_n^\circ-\an})'_b
\end{numequation}

gives rise to an isomorphism $\DHnn \hot_\DlHN V' \lra V_{\bbH_n^\circ-\an}$. The key result of this note is that one does not need to use the completed tensor product here. In other words:

\begin{theorem}\label{Main1} (cf. \ref{coad}) The map \ref{can}
is bijective.
\end{theorem}

Then we note that results of Emerton and Schneider-Teitelbaum imply that

\begin{theorem}\label{Main2} (cf. \ref{flatness}) The homomorphism \ref{can_morph} is flat. 
\end{theorem}

Combining these two results gives

\begin{theorem}\label{Main3} (cf. \ref{Exact}) The functor $V \rightsquigarrow V_{\bbH_n^\circ-\an}$ from the category of locally analytic admissible $H$-representations to the category of $E$-vector spaces is exact. 
\end{theorem}

Most probably the results in this note are well known to the experts. But the author has not seen \ref{Main1} or \ref{Main3} stated like this in literature.\footnote{After a first version of this paper was written and posted on arXiv, it was pointed out to us that the results of this paper are contained in the appendix of the preprint \cite{Em_JacquetII}. While the arguments given here and in the appendix of \cite{Em_JacquetII} are similar, we make heavier use of the results of \cite{S-T2} and place more emphasis on the interaction of Fr\'echet-Stein and weak Fr\'echet-Stein structures. Moreover, our treatment of good analytic subgroups is based on the concept of uniform pro-$p$ groups and thus offers an alternative and self-contained point of view.} It will be clear to the reader how much this paper owes to the profound work done in \cite{S-T2} and \cite{Em17}.

\vskip8pt

\textit{Notation and conventions.}
As above, let $\Qp \subseteq F \subseteq E$ be finite field extensions. We denote by $\cO_F$ the ring of integers of $F$. Further, we denote by $|.|$ the unique extension to $E$ (resp $F$) of the $p$-adic absolute value $|.|_p$ on $\bbQ_p$ normalized by $|p|_p = p^{-1}$.

\section{Preliminaries}
\begin{para}\label{weakFS} {\it Weak Fr\'echet-Stein algebras.} For general background on locally convex vector spaces over non-archimidean fields we refer the reader to \cite{NFA} and \cite{Em17}.
For a locally convex $E$-algebra $A$ the notion of a  weak Fr\'echet-Stein structure on $A$ has been defined in \cite[Def 1.2.6]{Em17}. We recall the definition here.

\vskip8pt

Let $A$ be a locally convex $E$-algebra, a \textit{weak Fr\'echet-Stein structure} on $A$ consists of the following data:

\begin{enumerate}\label{WFS}
\item A sequence $\{A_n\}_n$ of hereditarily complete, locally convex topological $E$-algebras $A_n$, for each $n\geq 1$.

\item A BH map\footnote{A continuous linear map between Hausdorff topological $E$-vector spaces $U\ra V$  is called BH if it admits a factorization of the form $U\ra T \ra V$, where $T$ is a Banach $E$-space \cite[Def. 1.1.13]{Em17}} $\tau^A_{n+1,n}: A_{n+1}\rightarrow A_{n}$, that is a continuous homomorphism of locally convex topological $E$ vectors spaces, for each $n\geq 1$.

\item An isomorphism of topological $A$-modules $A\cong \varprojlim_n A_n$, where each of the maps $A\rightarrow A_n$ has dense image. The right hand side is given a projective limit topology from the transition maps induced from part (2).
\end{enumerate}

\vskip8pt

Such a locally convex topological algebra $A$ with a choice of weak Fr\'echet-Stein structure is called a \textit{weak Fr\'echet-Stein algebra}. 
\end{para}

\begin{example} {\it The locally analytic distribution algebra of $\Zp$.} Let $D^\la(\Zp, E)$ be the strong dual of the space of locally analytic functions $\Zp \ra E$. We observe that

$$\ClaZp_{p^n\cm-\an}= \bigoplus_{a\in \Z/p^{n+1}\Z}\cO(a+p^n\cm) \;,$$

\vskip8pt

where 

$$\cO(a+p^n\cm):=\left\{\sum_n c_n(x-a)^n|\; c_n\in E, \lim_{n\ra \infty}c_nr^n=0\;\forall\;r< |p|^n \right\} \;.$$

\vskip8pt

By compactness of $\Zp$ it follows that 

$$\ClaZp = \varinjlim_{n\geq 0}\Big(\bigoplus_{a\in \Z/p^{n+1}\Z}\cO(a+p^n\cm)\Big) \;,$$ 

\vskip8pt

and hence by \cite[Prop. 16.10 (iii)]{NFA}

$$D^\la(\Zp,E) =\varprojlim_{n\geq 0}\Big(\bigoplus_{a\in \Z/p^{n+1}\Z}\cO(a+p^n\cm)\Big)'_b \;.$$ 

\vskip8pt

Where the spaces are topologized by the locally convex inductive and projective limit topology respectively with the natural restrictions (and their duals) as transition maps. Let us set 

$$A_n:= \Big(\bigoplus_{a\in \Z/p^{n+1}\Z}\cO(a+p^n\cm)\Big)'_b $$ 

We show in Cor.~\ref{pnan} that $A_n\cong \cO_{r_n}(X)^{\dagger}$, where 

$$\cO_{r_n}(X)^{\dagger}:=\Big\{\sum_{n=0}^\infty a_n T^n \; \Big| \; \lim_{n \ra \infty} a_n R^n  \\= 0 \text{\;for \;some\;} R  > r_n\Big\}$$ 

\vskip8pt

is the set of overconvergent functions on a closed disk of radius $r_n:=|p|^{\frac{1}{p^n(p-1)}}$.

\vskip8pt

Let us denote by $\partial$  the generator of the Lie algebra of $\Zp$, given by $\partial(f)= f'(0)$. Put $q^{(m)}_i=\left\lfloor \frac{i}{p^m}\right\rfloor$, and set 

$$D^{(m)}_n:=\left\{\sum_{i=0}^\infty b_i \frac{q_i^{(m)}!}{i!}(p^n\partial)^i \;\Big|\; b_i\in E,\, \lim_{i \ra \infty} b_i\ra 0\right\} \;.$$ 

\vskip8pt

Then, the map $\partial\ra \log(1+T)$, induces an isomorphism $\varinjlim_m D^{(m)}_n \car A_n$. Each of the $D^{(m)}_n$ are Noetherian Banach algebras and the locally convex inductive limit topology coincides under this identification with the natural topology of the space of overconvergent functions. Each of the natural restriction maps $A_{n+1}\ra A_n$ factors through $A_{n+1}\ra D_{n}^{(0)}\ra A_n$. It can then be shown that the isomorphism $D^\la(\Zp,E) \cong \varprojlim_{n}A_n$ defines a weak Fr\'echet-Stein structure on $D^\la(\Zp,E)$, and the isomorphism $D^\la(\Zp,E)\cong \varprojlim_{n}D_n^{(0)}$ defines a Fr\'echet-Stein structure (whose definition is recalled below) on $D^\la(\Zp,E)$. \hfill $\blacksquare$

\end{example}

\vskip8pt

Next we recall the notion of coadmissible modules over weak Fr\'echet-Stein algebras and the stronger notion of Fr\'echet-Stein algebras.

\vskip8pt

\begin{para}\label{coadmissible} {\it Coadmissible modules.}  Let $A$ be a weak Fr\'echet-Stein algebra with a choice of weak Fr\'echet-Stein structure $A\cong\varprojlim_{n}A_n$. A locally convex topological $A$-module $M$ is called \textit{coadmissible} \cite[Def. 1.2.8]{Em17} (with respect to the chosen weak Fr\'echet-Stein structure of $A$) if there exist the following data:
\begin{enumerate}
\item A sequence $\{M_n\}_{n \ge 1}$ of finitely generated topological $A_n$-modules, for each $n\geq 1$.

\item An isomorphism $A_n \hot_{A_{n+1}}M_{n+1}\cong M_n $, for each $n\geq 1$.

\item An isomorphism of topological $A$-modules $M\cong \varprojlim_{n}M_n$, where the right hand side is given a projective limit structure and topology from the transition maps induced from part (2).
\end{enumerate}
\end{para}

\vskip8pt

\begin{rem}
The completed tensor product in \ref{coadmissible} (2) is defined as follows, cf. \cite[after 1.2.3]{Em17}. If $V$ and $W$ are two locally convex $E$-vector spaces then the locally convex projective tensor product topology on $V \ot_E W$ is the universal topology for jointly continuous bilinear maps $V \times W \ra  U$ of locally convex $E$-vector spaces. We let $V \ot_{E,\pi} W$ denote $V \ot_E W$ equipped with this projective tensor product topology. Let $A \ra B$ be a continuous homomorphism of locally convex topological $E$-algebras and let $M$ be a locally convex topological $A$-module. If $B \ot_A M$ is endowed by the quotient topology obtained by regarding it as a quotient of $B \ot_{E,\pi} M$, then $B \ot_A M$ is a locally convex topological $B$-module, cf. \cite[1.2.3]{Em17}. We let $B \hot_A M$ be the completion of the locally convex $E$-vector space $B \ot_A M$, cf. \cite[7.5]{NFA}. By \cite[1.2.2]{Em17} the completion $B \hot_A M$ carries a canonical structure of a locally convex topological $B$-module. \hfill $\blacksquare$
\end{rem}

\begin{para}\label{independence} Given a weak Fr\'echet-Stein algebra $A$, with a weak Fr\'echet-Stein structure $A \car \varprojlim_n A_n$, and a coadmissible $A$-module $M$, a sequence $\{M_n\}_n$ as in \ref{coadmissible} (1) is referred to as an $\{A_n\}_n$-sequence for $M$. We let $\pr^{M_\bullet}_n: M \ra M_n$ be the map with is the composition of the chosen isomorphism $M \car \varprojlim_n M_n$ followed by the canonical projection $\varprojlim_n M_n \ra M_n$. We also recall that if $M$ is a coadmissible module with respect to a certain weak Fr\'echet-Stein structure then it is coadmissible with respect to any weak Fr\'echet-Stein structure on the same algebra \cite[Prop. 1.2.9]{Em17}.
\end{para}
 
\vskip8pt

\begin{para}\label{FS-algs} {\it Fr\'echet-Stein algebras.} A locally convex $E$-algebra $A$ is said to have a \textit{Fr\'echet-Stein structure} \cite[Def. 1.2.10]{Em17} if it is endowed with a weak Fr\'echet-Stein structure $A \car \varprojlim_n B_n$ with the additional requirements that each $B_n$ is left Noetherian and each map $B_{n+1} \ra B_n$ is right flat. Such an algebra equipped with a Fr\'echet-Stein structure is called a \textit{Fr\'echet-Stein algebra}.
We begin with recalling the following useful theorems.
\end{para}

\begin{theorem}\cite[Cor. 3.1]{S-T2}\label{tensorbijection}
Let $A$ be a Fr\'echet-Stein algebra, equipped with a Fr\'echet-Stein structure $A \car \varprojlim_n B_n$. Let $M \car \varprojlim_n M_n$ be a coadmissible $A$-module, where $\{M_n\}_n$ is a $\{B_n\}_n$-sequence. Then the map

$$\id_{B_n} \ot \pr^{M_\bullet}_n: B_n \ot_A M \lra B_n \ot_{B_n} M_n = M_n $$ 

\vskip8pt

is an isomorphism for all $n \ge 1$.
\end{theorem}

\begin{rem}\label{STrem}
The argument in \cite[Cor. 3.1]{S-T2} crucially uses that the $B_n$ are (left) Noetherian. We will later adapt the argument and generalize the preceding theorem in Thm. \ref{ThWFS}.
\hfill $\blacksquare$ 
\end{rem}

When $A$ is a Fr\'echet-Stein algebra, equipped with a weak Fr\'echet-Stein structure $A \car \varprojlim_n A_n$, we have the following

\begin{theorem}
\cite[Thm. 1.2.11 (i)]{Em17}\label{lemco}
Let $A$ be a Fr\'echet-Stein algebra, equipped with a weak Fr\'echet-Stein structure $A \car\varprojlim_n A_n$. Let $M\car \varprojlim_n M_n$ be a coadmissible $A$-module, where $\{M_n\}_n$ is an $\{A_n\}_n$-sequence. Then the map 
 
$$\id_{A_n} \ot \pr^{M_\bullet}_n: A_n \ot_A M \lra A_n \ot_{A_n} M_n = M_n$$ 

\vskip8pt

induces an isomorphism  $A_n \hot_A M \ra M_n$ for all $n\geq 1$.
\end{theorem}

We will also need the following properties of algebras of compact type.
\begin{lemma}\label{cptype}\cite[Prop. 5.1.1]{PSS}
Let $C\car\varinjlim_m C^{(m)}$ be an $E$-algebra of compact type with Noetherian Banach algebras $C^{(m)}$ over $E$.

\vskip8pt

(i) A finitely generated module $P$ over $C$ has a a unique compact type topology $P\car\varinjlim_m P_m$, called the canonical topology, such that $C\times P\ra P$ is a continuous map. We can choose each $P_m$ to be a finitely generated module over $C^{(m)}$.

\vskip5pt

(ii) A finitely generated $C$-submodule of $P$ is closed in the canonical topology.

\vskip5pt

(iii) Any $C$-linear map between finitely generated $C$-modules is continuous and strict with respect to the canonical topology of part (i).
\end{lemma}

Our main goal in this section is to prove the

\begin{theorem}\label{ThWFS}
Let $A$ be a weak Fr\'echet-Stein algebra with a weak Fr\'echet-Stein structure $A\car\varprojlim_n A_n$, where each $A_n$ is assumed to be an $E$-algebra  of compact type. Suppose $\{B_n, \tau^B_{n+1,n}\}_{n\geq 1}$ is a projective system of left Noetherian Banach $E$-algebras with the following properties

\vskip8pt

\begin{enumerate}
    \item There is a morphism $(i_n)_n: \{B_n,\tau^B_{n+1,n}\}_n \ra \{A_n,\tau^A_{n+1,n}\}_n$ of projective systems, such that all $B_n \xrightarrow{i_n} A_n$ are continuous morphisms of $E$-algebras. In particular, $i_n \circ \tau^B_{n+1,n} = \tau^A_{n+1,n} \circ i_{n+1}$.
    
    \vskip5pt
    
    \item The map 
    
$$i := \varprojlim_n i_n: \varprojlim_n B_n \lra \varprojlim_n A_n \car A$$ 

\vskip8pt

\noindent is an isomorphism of topological $E$-algebras, with the additional property that the sequence $\{B_n, \tau^B_{n+1,n}\}_{n\geq 1}$ is a Fr\'echet-Stein structure on $A$.\footnote{cf. the following remark}
\end{enumerate}

\vskip8pt
 
Let $M\car \varprojlim_n M_n$ be a coadmissible module with respect to the weak Fr\'echet-Stein structure given by $\{A_n\}_n$, with each $M_n$ a finitely generated $A_n$-module equipped with the canonical compact type topology. Then the canonical map  

$$\phi_n := \id_{A_n} \ot \pr^{M_\bullet}_n:  A_n \ot_A M \lra  M_n$$ 

\vskip8pt

is a bijection.
\end{theorem}

\begin{rem}\label{implications} When we assume in the statement of the preceding theorem that the sequence $\{B_n, \tau^B_{n+1,n}\}_{n\geq 1}$ is a Fr\'echet-Stein structure on $A$, then this means that the following conditions are supposed to be satisfied:

\vskip8pt

\begin{enumerate} 
\item For every $n \in \Ne$ the transition map $\tau^B_{n+1,n}: B_{n+1} \ra B_n$ makes $B_n$ a right flat $B_{n+1}$-module.

\vskip5pt

\item  Each map $\pr^{B_\bullet}_n: A \ra B_n$, obtained from composing $i^{-1}: A \ra \varprojlim_n B_n$ with the canonical projection map $ \varprojlim_n B_n \ra B_n$, has dense image. 
\end{enumerate}

\vskip8pt

As each $B_n$ is assumed to be a Banach algebra, it is hereditarily complete, cf. \cite[after Def. 1.1.39]{Em17}.  \hfill $\blacksquare$
\end{rem}

\begin{proof} As mentioned in \ref{STrem}, this is an adaptation of the argument in \cite[Cor. 3.1]{S-T2}. Instead of assuming that the $A_n$ are Noetherian we generalize it for algebras of compact type using lemma \ref{cptype}. We also mention that the proof of the injectivity of $\phi_n$ given below is essentially the same as the proof in \cite[Cor. 3.1]{S-T2} with additional details furnished. 

\vskip8pt

(i) {\it Surjectivity.} We consider $B_n$ as an $A$-module via the map $\pr^{B_\bullet}_n: A \ra B_n$ introduced in \ref{implications}. By \ref{independence} $M$ is also a coadmissible module for the Fr\'echet-Stein structure $A \car \varprojlim_n B_n$ on $A$. Let $\{M_n^B\}_n$ be a $\{B_n\}_n$-sequence for $M$. By \ref{tensorbijection} the map 

$$\id_{B_n} \ot \pr^{M^B_\bullet}_n: B_n \ot_A M \lra  B_n \otimes_{B_n} M^B_n = M^B_n$$ 

\vskip8pt

is an isomorphism, and $B_n \ot_A M$ is hence a finitely generated $B_n$-module (because $M^B_n$ is a finitely generated $B_n$-module). Furthermore, $A_n\ot_A M \cong A_n\ot_{B_n} B_n\ot_A M$. Consequently $\im(\phi_n)$ is a finitely generated $A_n$-submodule of $M_n$. Therefore, $\im(\phi_n)$ is closed in the canonical topology of $M_n$ (Lemma \ref{cptype} (ii), here $C= A_n, P= M_n$). By  \ref{lemco}, the map  $\phi_n = \id_{A_n} \ot \pr^{M_\bullet}_n: A_n \ot_A M \ra M_n$ induces an isomorphism $A_n \hot_A M \car M_n$, and the image of $\phi_n$ is thus dense in $M_n$.  Density and closedness of the image shows that $\phi_n$ is surjective.

\vskip8pt

(ii) {\it Injectivity.} We will use the surjectivity result from above to prove injectivity of $\phi_n$. Suppose $x:= b_1\otimes x_1+ \ldots+ b_r\otimes x_r \in A_n \ot_A M$ is such that 

$$(\id_{A_n} \otimes \pr^{M_\bullet})(x) = b_1.\pr^{M_\bullet}_{n}(x_1) + \ldots + b_r.\pr^{M_\bullet}_{n}(x_r) = 0$$ 

\vskip8pt

in $M_n$. For $\ell \geq n$ we consider the map

$$\psi_\ell: A_\ell^r \lra M_\ell \;, \;\; (a_1,\ldots,a_r) \longmapsto a_1.\pr^{M_\bullet}_\ell(x_1) + \ldots + a_r.\pr^{M_\bullet}_\ell(x_r) \;.$$ 

\vskip8pt

Let us denote the kernel of this map by $N_\ell$. Thus, for each $\ell \geq n$ we have an exact sequence

$$0\rightarrow N_\ell\hookrightarrow A^r_\ell\stackrel{\psi_\ell}{\longrightarrow} M_\ell \;.$$

\vskip8pt

Setting $N = \varprojlim_\ell N_\ell$ and passing to the projective limit (which is left exact), we get an exact sequence 

$$0\rightarrow N\hookrightarrow  A^r \stackrel{\psi}{\longrightarrow}  M \;,$$

\vskip8pt

where the rightmost map is given by  $(a_1,\ldots,a_r) \mapsto a_1.x_1 + \ldots + a_r.x_r$. The category of coadmissible modules is closed under taking kernels \cite[Thm. 1.2.11 (ii)]{Em17}, \cite[Cor. 3.4 (ii)]{S-T2}. In particular $N$ is a coadmissible module with the corresponding $\{A_\ell\}_\ell$-sequence $\{N_\ell\}_\ell$. Further, $N_\ell$ is a finitely generated module over $A_\ell$ (by definition of ``coadmissible module") and it then follows from the above argument that the map from $\id_{A_n} \ot \pr^{N_\bullet}_n: A_n \ot_A N \rightarrow N_n$ is surjective. We can then find elements $(c_1^{(1)},\ldots,c_r^{(1)}),\ldots,(c_1^{(s)},\ldots,c_r^{(s)})$ of $N\subseteq A^r$ whose images under $\pr_n^{N_\bullet}: N\ra N_n$ generate $N_n$ as an $A_n$-module. Recall the element $x= b_1\ot x_1+\cdots+ b_r\ot x_r$ from above. By definition of $x$ we see that $(b_1,\ldots,b_r) \in \ker(\psi_n) = N_n$. Thus there exists $f_1,\ldots,f_s \in A_n$ such that 

$$(b_1,\ldots,b_r)= f_1.(\pr_n^{A_\bullet}(c_1^{(1)}), \ldots, \pr^{A_\bullet}_n(c_r^{(1)}))+ \ldots + f_s. (\pr^{A_\bullet}_n(c_1^{(s)}), \ldots, \pr^{A_\bullet}_n(c_r^{(s)})) $$

\vskip8pt

It follows that

\begin{align*}
    b_1\otimes x_1+ \ldots+ b_r\otimes x_r &= \sum_{j=1}^{r}(b_j\otimes x_j) = \sum_{j=1}^{r}\Big(\sum_i f_i.\pr^{A_\bullet}_n(c_j^{(i)})\Big)\otimes x_j\\
    & = \sum_{j=1}^{r}\Big(\sum_{i}f_i\otimes c_j^{(i)}x_j\Big) = \sum_{i}f_i\otimes \Big(\sum_j c_j^{(i)}x_j\Big) \\
    & = \sum f_i\otimes 0 = 0
\end{align*}

\vskip8pt

The third equality follows from $a_n.\pr^{A_\bullet}_n(b) \ot m = a_n \ot b.m \in A_n\ot_A M$. So, $\phi_n: A_n \ot_A M \ra M_n$ is also injective.
\end{proof}

\section{Uniform pro-\texorpdfstring{$p$}{} groups, \texorpdfstring{$F$}{}-uniform groups, and rigid analytic groups}

\begin{para}{\it Uniform pro-$p$ groups.} In this section, $H$ will always denote a pro-$p$ group, whose group law we always write multiplicatively. $H$ is called \textit{powerful} \cite[Def. 3.1. (i)]{DDMS} if $p$ is odd and $H/\overline{H^p}$ is abelian or $p=2$ and $H/\overline{H^4}$ is abelian. The lower $p$-series $(P_i(H))_{i \ge 1}$ of $H$ is defined by

$$P_1(H)= H \;, \;\; P_{i+1}(H):=\overline{P_i(H)^p[P_i(H),H]}$$

\vskip8pt

where $[P_i(H), H]$ is the commutator of these two groups within $H$. 

\vskip8pt

A powerful pro-$p$ group $H$ is called \textit{uniformly powerful pro-$p$} or \textit{uniform pro-$p$} \cite[Def. 4.1]{DDMS} if 
\begin{enumerate}
    \item $H$ is finitely generated as  a pro-$p$ group.
    
    \item  For all $i$ the $p$-th power map $x \mapsto x^p$ induces an isomorphism of abelian groups
    
    $$P_i(H)/P_{i+1}(H) \lra P_{i+1}(H)/ P_{i+2} (H)$$

\end{enumerate}

\vskip8pt

If $H$ is a topologically finitely generated powerful pro-$p$ group, then $P_{i+1}(H) = \{ x^{p^i} \midc x \in H\}$ for every $i \ge 0$, cf. \cite[3.6 (iii)]{DDMS}. If $H$ is uniform pro-$p$, then the map $x \mapsto x^{p^i}$ is a homeomorphism $H \ra P_{i+1}(H)$, in particular, each element $z \in P_{i+1}(H)$ has a unique $p^i$-th root in $H$ which we denote by $z^{p^{-i}}$. 
\end{para}

\vskip8pt

\begin{para}\label{ZpLie} {\it The $\Zp$-Lie algebra associated to a uniform pro-$p$ group \cite[sec. 4.3, 4.5]{DDMS}.} In this paragraph, let $H$ be a uniform pro-$p$ group, and set $H_n = P_{n+1}(H) = \{x^{p^n} \midc x \in H\}$ for $n \ge 0$. For $x,y \in H$, the element $x^{p^n}y^{p^n}$ lies in $H_n$ and hence has a unique $p^n$-th root $(x^{p^n}y^{p^n})^{p^{-n}}$, as recalled above. By the discussion in \cite[sec. 4.3]{DDMS}, the limit 

$$x+y =\lim_{n \ra \infty} (x^{p^n}y^{p^n})^{p^{-n}}$$

\vskip8pt

exists in $H$ and gives $H$ the additional structure of an {\it abelian} group. The multiplication $\Z \times H \ra H$ extends continuously to $\Zp \times H \ra H$, and provides $(H,+)$ with the structure of a finitely generated free $\Zp$-module, whose rank is equal to the minimal number of topological generators of $H$.

\vskip8pt

As usual, we denote by $[x,y]_H = xyx^{-1}y^{-1}$ the commutator of $x,y \in H$. Note that  $[x^{p^n}, y^{p^n}]_H \in [H_n, H_n] \subseteq H_{2n+1}$ for $x,y \in H$. We then endow $H$ with a Lie bracket by setting

$$[x, y]_\Lie := \lim_{n \ra \infty} [x^{p^n}, y^{p^n}]_H^{p^{-2n}} \;,$$

\vskip8pt

where the limit exists by \cite[4.28]{DDMS}. By \cite[Thm. 4.30]{DDMS} the triple $L_H := (H, \, + \,, \, [\, , \, ]_\Lie)$ is a $\Zp$-Lie algebra which is \textit{powerful} in the sense of \cite[sec. 9.4]{DDMS}, i.e.,

$$[L_H, L_H] \subseteq p^\eps L_H \;,$$ 

\vskip8pt

where $\epsilon= 1$ if $p$ is odd and is equal to $2$ if $p=2$. We denote by $\exp_H$ the map $L_H \ra H$ which is the identity map on the underlying sets. The $\Qp$-Lie algebra $L_H \otimes_\Zp \Qp$ is canonically isomorphic to the Lie algebra $\Lie(H)$ of $H$ as a $p$-adic Lie group,\footnote{as defined in, for instance, \cite[ch. VII]{Schneider_Lie}} and the exponential map $\Lie(H) \dashrightarrow H$, which is only defined on a sufficiently small lattice in $\Lie(H)$, is equal to $\exp_H$ (as defined here), when restricted to $L_H \sub \Lie(H)$. 

\vskip8pt

By \cite[9.10]{DDMS} the functor $H \rightsquigarrow L_H$ defines an equivalence of between the category of uniform pro-$p$ groups and the category of powerful $\Zp$-Lie algebras. 
\end{para}

\vskip8pt

\begin{para} {\it From powerful $\Zp$-Lie algebras to uniform pro-$p$ groups.} The Baker-Campbell-Hausdorff (BCH) formula \cite[6.28]{DDMS} is a formal series

$$\Phi(X,Y):=\sum_{n=1}^{\infty}u_n(X,Y)\;,$$

\vskip8pt 

whose terms $u_n(X,Y)$ are in the free associative $\Q$-algebra generated by two variables $X$ and $Y$. If we denote by $[f,g] = fg-gf$ the commutator of two elements in this ring, then we have the following formulas 

$$u_1(X, Y) = X + Y\;, \;\; u_2(X, Y )= \frac{1}{2}[X, Y]\;,$$

$$u_n(x,y) = \sum_{\mbox{{\tiny $\begin{array}{c}e = (e_1, \ldots, e_n) \in \Ne^n\\ e_1+ \ldots + e_n = n-1 \end{array}$}}} q_e [X, Y]_e \;.$$

\vskip8pt

Here 

$$[X, Y]_e = \Big[\ldots [X, [\underbrace{Y,\ldots, Y}_{e_1}, \underbrace{X,\ldots, X}_{e_2}, \ldots] \ldots\Big]$$ 

\vskip8pt

denotes an iterated commutator, and the $q_e$ are certain constants in $\Q$. If $L$ is a powerful $\Zp$-Lie algebra, then for any $x,y \in L$ each term $u_n(x,y)$, which a priori is in $L \otimes \Q$, lies actually in $L$, and the series $\Phi(x,y)$ converges in $L$. Moreover, the map 

$$L \times L \lra L \;, \;\; (x,y) \mapsto x \cdot_\Phi y := \Phi(xy)$$ 

\vskip8pt

defines a group structure on $L$. By \cite[9.8]{DDMS} the pair $(L, \, \cdot_\Phi \,)$ is a uniform pro-$p$ group. In fact the functor from the category of powerful $\Zp$-Lie algebras to the category of uniform pro-$p$ groups, $(L, [\,,\,]) \rightsquigarrow (L, \, \cdot_\Phi \,) $ is a quasi-inverse to the functor  $H \rightsquigarrow L_H $, cf. \cite[9.10]{DDMS}.
\end{para}

\vskip8pt

\begin{dfn}
A locally $F$-analytic group\footnote{cf. \cite[sec. 13]{Schneider_Lie}, where those groups are called ''Lie groups over $F$''} $H$ is called \textit{$F$-uniform pro-$p$} if the following conditions hold:

\begin{enumerate}
    \item $H$ is a uniform pro-$p$ group.
    
    \vskip5pt
    
    \item The $\mathbb{Z}_p$-Lie algebra $L_H$, which embeds canonically into $\Lie(H) = L_H \otimes_\Zp \Qp$ (cf. end of sec. \ref{ZpLie}), is stable under multiplication by $\cO_F$ on $\Lie(H)$. (In particular, $L_H$ has the structure of an $\cO_F$-module.)
\end{enumerate}
\end{dfn}

\vskip8pt
 
 \begin{para}\label{rigid} {\it Rigid analytic groups associated to $F$-uniform groups.} Let $H$ be an $F$-uniform pro-$p$ group. The $\cO_F$-module $L_H$ is free of finite rank $d$ over $\cO_F$. We choose an $\cO_F$-basis $\der_1, \ldots, \der_d$ of $L_H$. For $x_1, \ldots, x_d , y_1, \ldots, y_d \in \cO_F$ we can write
 
 $$\Phi(x_1\der_1+ \ldots + x_d \der_d, y_1\der_1+ \ldots + y_d \der_d) = \sum_{j=1}^d \Phi_j(x_1, \ldots, x_d, y_1, \ldots,y_d)\der_j$$
 
 \vskip8pt
 
 where each $\Phi_j$ is a power series in the ring $\cO\langle x_1, \ldots, x_d, y_1, \ldots,y_d \rangle$ of strictly convergent power series in $x_1, \ldots, x_d, y_1, \ldots,y_d$. Consider $L_H^\vee := \Hom_{\cO_F}(L_H,\cO_F)$ and let $(\der_j^\vee)_j$ be the dual basis to $(\der_j)_j$, i.e., $\der_j^\vee(\der_i) = \delta_{i,j}$. Then the $p$-adic completion $\Sym_{\cO_F}(L_H^\vee)^\wedge$ of the symmetric algebra $\Sym_{\cO_F}(L_H^\vee) \cong \cO_F[\der_1^\vee, \ldots, \der_d^\vee]$ is ismormorphic to $\cO_F\langle \der_1^\vee, \ldots, \der_d^\vee \rangle$. Via those isomorphisms we define a comultiplication
 
 \begin{numequation}\label{comult}\Sym_{\cO_F}(L_H^\vee)^\wedge \lra \Sym_{\cO_F}\Big((L_H \oplus L_H)^\vee\Big)^\wedge
 \end{numequation}
 
by sending $\der_j^\vee$ to $\Phi_j(\der^\vee_1, \ldots, \der^\vee_d, \der^\vee_1, \ldots,\der^\vee_d)$.
This gives the affinoid space 

$$\bbH = \Sp\left(\Sym_{\cO_F}(L_H^\vee)^\wedge \ot_{\cO_F} F\right)$$

\vskip8pt

the structure of a rigid analytic group, which is independent of the choice of the basis $(\der_j)_j$ made above. More generally, if $\vpi$ is a uniformizer of $F$, then we denote by $\bbH_n \sub \bbH$ the rigid analytic affinoid subgroup that one obtains by replacing $L_H$ by $\vpi^n L_H$. If $\exp_H: L_H \ra H$ is the exponential map as in \ref{ZpLie}, then the group of $F$-valued points of $\bbH_n$ is $\exp_H(\vpi^n L_H)$.

\vskip8pt

 It follows from the discussion in \cite[p. 100]{Em17} that for any $n \ge 0$ the rigid analytic group $\bbH_n$ is a good analytic subgroup of $H$, using the terminology introduced by Emerton in \cite{Em17}. Thus the subsequent constructions of distribution algebras, recalled in the next section, apply in our setting. 
 \end{para}
 
 \vskip8pt
 
 \begin{para}\label{wideopen} {\it The wide open groups $\bbH_n^\circ$.} Because the comultiplication map \ref{comult} is defined integrally, it equips the formal scheme $\frH = \Spf\Big(\Sym_{\cO_F}(L_H^\vee)^\wedge\Big)$ with the structure of a formal group scheme over $\Spf(\cO_F)$. Let $\frH^\circ$ be the completion of $\frH$ along its unit section $e_\frH: \Spf(\cO_F) \ra \frH$. This is a affine formal group scheme over $\Spf(\cO_F)$ whose corresponding $\cO_F$-algebra $\cO(\frH^\circ)$ is the completion of $\Sym_{\cO_F}(L_H^\vee)$ with respect to the augmentation ideal $I := \ker\Big(\Sym_{\cO_F}(L_H^\vee) \xrightarrow{e_\frH^*} \cO_F\Big)$. Using the basis $(\der_j^\vee)_j$ introduced above, one has an isomorphism $\cO(\frH^\circ) \cong \cO_F[[\der_1^\vee, \ldots, \der_d^\vee]]$. We then let $\bbH^\circ$ be the rigid analytic group associated to the group scheme $\frH^\circ$ (by Berthelot's construction \cite[sec. 7]{deJongCrys}). In the same way we construct the rigid analytic group $\bbH_n^\circ$ starting from the $\Zp$-Lie algebra $\vpi^n L_H$, which we call {\it wide open}. One has $\bbH_{n+1} \sub \bbH_n^\circ \sub \bbH_n$ and $\bbH_n^\circ(F) = \bbH_{n+1}(F)$. 

 \end{para}
 
\section{Distribution Algebras (summary of results)} We recall that $E/F$ denotes a  finite field extension. Given a compact, locally $F$-analytic group $G$ we let $C^\la(G,E)$ be the locally convex $E$-vector space of locally $F$-analytic functions $G \ra E$. The space $C^\la(G,E)$ is of compact type and is a locally $F$-analytic  representation of $G$ via the right regular action of $G$, i.e., $g. f(x)= f( x. g^{-1})$ (cf. \cite[sec. 2]{S-T} for details of the construction and the topology on the space). The strong dual $D^\la(G,E):= (C^\la(G,E)'_b)$ of $C^\la(G,E)$ is called the locally $F$-analytic distribution algebra of $G$ (with coefficients in $E$). $D^\la(G,E)$ is an $E$-Fr\'echet space and it has an $E$-algebra structure with convolution product as multiplication. For future use we note that for each $g \in G$ we have the delta distribution $\delta_g \in D^\la(G,E)$ defined by $\delta_g(f)= f(g)$. 

\vskip8pt

From now on we fix an $F$-uniform pro-$p$ group. Given a wide open rigid analytic group $\bbH_n^\circ$ as above, we let $\cD^\an(\bbH_n^\circ,E) = \Hom^\cont_F(\cO(\bbH_n^\circ),E)'_b$ be the analytic distribution algebra of Emerton, cf. \cite[2.2.2]{Em17}. Moreover, we consider the $E$-algebra 

$$D(\bbH^\circ_n, H) := \big( C^{\la}(H,E)_{\bbH_n^\circ-\an}\big)'_b$$ 

which is the strong dual of the $\bbH_n^\circ$-rigid analytic vectors of $C^{\la}(H,E)$. It is shown in \cite[ proof of 5.3.1]{Em17} that


$$D(\bbH^\circ_n, H) = \oplus_{h\in H/H_{n+1}}\;\delta_h \cdot \cD^\an(\bbH_n^{\circ}, E)\,,$$

\vskip8pt

where $H_{n+1} = \bbH_{n+1}(F) = \bbH_n^\circ(F)$. As in \ref{rigid}, we choose an $\cO_F$-basis  $\partial_1,\ldots,\partial_d$ of the Lie algebra $L_H$ corresponding to $H$. As in \cite[p. 102]{Em17} we have the following $\cO_E$-subalgebra of the universal enveloping algebra $U(L_H)$

$$A^{(m)}:=\Big\{\sum_I b_I \frac{q_{i_1}^{(m)}!\ldots q_{i_d}^{(m)}!}{i_1!\ldots i_d!}\partial_1^{i_1}\ldots \partial_d^{i_d} \,\Big|\; b_I\in \cO_E,\\
b_I= 0 \; \text{for\; almost\; all}\; I\Big\}$$

\vskip8pt

where $q^{(m)}_i=\left\lfloor \frac{i}{p^m} \right\rfloor$. Let $\widehat{A^{(m)}}$ be the $p$-adic completion of $A^{(m)}$. We recall the definitions of \cite[p. 97, 107]{Em17}

$$\cD^\an(\bbH_n^{\circ}, E)^{(m)} := E\ot_{\cO_E}\widehat{A^{(m)}}$$

\vskip8pt

and


$$\begin{array}{rl} D(\bbH^\circ_n, H)^{(m)} & := E[H] \ot_{E[H_{n+1}]} \Danhn^{(m)} \\
&\\
& =\oplus_{h\in H/H_{n+1}} \; \delta_h \cdot \cD^\an(\bbH_n^{\circ}, E)^{(m)}
\end{array}$$


\vskip8pt

It is shown in the proof of \cite[5.3.19]{Em17} that for every $n\geq 0$ there exists a $m_n \in \Ne$ such that the natural map $D(\bbH_{n+1}^\circ, H) \ra D(\bbH_n^{\circ},H)$ factors as 

\begin{numequation}\label{inclusions}
D(\bbH_{n+1}^\circ,H) \lra D(\bbH_n^\circ,H)^{(m_n)} \lra D^\an(\bbH_n^\circ, H) \;.
\end{numequation}

Moreover we have the following important results.

\begin{theorem}\label{structure1}\cite[5.3.11, 5.2.6, and remark before 5.2.6]{Em17} (i) $D(\bbH_n^{\circ}, H)^{(m)}$ is a Noetherian Banach algebra for each $m, n \in \Ne$.

\vskip8pt

(ii) There exists a natural isomorphism of topological $E$-algebras 

$$\varinjlim_m D(\bbH_n^{\circ}, H)^{(m)} \car  D(\bbH_n^\circ, H) \,.$$

\vskip8pt

This isomorphism makes $ D(\bbH_n^\circ, H)$  an $E$-algebra of compact type. \qed
\end{theorem}

\vskip8pt

\begin{theorem}\label{structure2} (i) For an $F$-uniform pro-$p$ group $H$ the distribution algebra $D^\la(H, E)$ is a weak Fr\'echet-Stein algebra with a weak Fr\'echet-Stein structure 

$$D^\la(H, E)\car \varprojlim_n D(\bbH_n^{\circ}, H)\,.$$

\vskip5pt

(ii) The inclusions $D(\bbH_n^\circ,H)^{(m)} \hra D^\an(\bbH_n^\circ, H)$ are flat for all $n$ and $m$.

\vskip5pt

(iii) The inclusions in (ii) induce an isomorphism

$$D^\la(H, E) \car \varprojlim_n D(\bbH_n^{\circ}, H)^{(m_n)} \,,$$

\vskip8pt

which equip $D^\la(H, E)$ with a Fr\'echet-Stein structure. 
\end{theorem}

\begin{proof} (i) \cite[5.3.1]{Em17}.

\vskip8pt

(ii) This is stated after equation (5.3.20) in the proof of \cite[5.3.19]{Em17}.

\vskip8pt

(iii) \cite[5.3.19]{Em17}. 
\end{proof}

\vskip8pt

As a consequence of the previous theorem we obtain the following flatness results.

\begin{theorem}\label{flatness}
(i) For every $n \ge 0$ the homomorphism of $E$-algebras 

$$\pr^{\DlHN_\bullet}_n:\DlHN \lra \DHnn^{(m_n)} \;,$$ 
\vskip8pt

obtained from the canonical projection $\DlHN \ra D(\bbH_{n+1}^\circ,H)$, followed by the inclusion $D(\bbH_{n+1}^\circ,H) \hra D(\bbH_n^\circ,H)^{(m_n)}$ in \ref{inclusions}, is flat.

\vskip5pt

(ii)  The natural map $$\DlHN\ra \DHnn$$ is flat for every $n\geq 1$.
\end{theorem}

\begin{proof}
(i) Because $D^\la(H, E) \car \varprojlim_n D(\bbH_n^{\circ}, H)^{(m_n)}$ gives $D^\la(H, E)$ a Fr\'echet-Stein structure, cf. \ref{structure2} (iii), the map $\pr^{\DlHN_\bullet}_n$ is flat by \cite[Rem. 3.2]{S-T2}.

\vskip8pt

(ii) This follows from part (i) and \ref{structure2} (ii). 
\end{proof}

\section{Application to rigid analytic vectors}

\begin{para}\label{rigvectors} {\it Admissible representations and rigid analytic vectors.} Let $G$ be a locally $F$-analytic group, and let $V$ be an admissible locally analytic representation of $G$ on an $E$-vector space $V$, as introduced in \cite{S-T2}. This means that the dual space $M := V'_b$, equipped with the strong topology, is a coadmissible module over the locally analytic distribution algebra $\DlHN$ for any compact open subgroup $H \sub G$. In the following we will assume that $H$ is an $F$-uniform pro-$p$ group.\footnote{Any such $G$ has a fundamental system of open neighborhoods of the identity element which are $F$-uniform pro-$p$ groups.} Whereas the definition of ''coadmissible'' in \cite{S-T2} is based on the key result that $\DlHN$ is a Fr\'echet-Stein algebra, we are here rather interested in the weak Fr\'echet-Stein structure given by $D^\la(H, E)\car \varprojlim_n D(\bbH_n^{\circ}, H)$, cf. \ref{structure1}. In this case, a $\{D(\bbH_n^{\circ}, H)\}_n$-sequence for $M$, cf. \ref{independence}, is given by

$$M_n = \left(V_{\bbH_n^\circ-\an}\right)'_b$$

\vskip8pt

as follows from \cite[6.1.19]{Em17} and its proof. We note here, that even though there is a natural continuous injection from $V_{\bbH_n^\circ-\an} \ra V$, the space $V_{\bbH_n^\circ-\an}$ is topologized as a closed subspace of $C^\an(\bbH, V)$, where $C^\an(\bbH, V)$ is the space of $V$-valued $F$-rigid analytic functions on $\bbH$, cf. \cite[3.3.1, 3.3.13]{Em17}. Usually this intrinsic topology on  $V_{\bbH_n^\circ-\an}$  does not conincide with the subspace topology induced on it as a subspace of $V$.

Therefore, by \ref{lemco}, the canonical map $D(\bbH_n^{\circ}, H) \hot_{D^\la(H, E)} M \ra M_n$ is an isomorphism of $D(\bbH_n^{\circ}, H)$-modules. The next result shows that the same is true for the ordinary tensor product instead of the completed tensor product. 
\end{para}

\begin{theorem}\label{coad} Let $M = V'_b$ be a coadmissible $D^\la(H, E)$-module, and let $M_n = \left(V_{\bbH_n^\circ-\an}\right)'_b$ be as above. Then the canonical map 

$$\DHnn\otimes_{D^\la(H, E)} M \lra M_n \;,$$

\vskip8pt

is an isomorphism of topological $\DHnn$ modules for all $n \ge 1$.
\end{theorem}

\begin{proof}
Set $A= \DlHN$, $A_n = \DHnn$, and $B_n = \DHnn^{(m_n)}$, with $m_n$ as in \ref{inclusions}. Then, because of \ref{structure1} and \ref{structure2} we can apply \ref{ThWFS}, and deduce the assertion above.
\end{proof}

\begin{theorem}\label{Exact}
Let $H$ be an $F$-uniform pro-$p$ group, and let $\bbH_n^\circ$ be defined as in \ref{wideopen}. The functor from the category of locally analytic admissible representations of $H$ over $E$ to the category of $E$-vector spaces given by 

$$V \rightsquigarrow V_{\bbH_n^{\circ}-\an}$$

is exact for every $n \ge 1$.
\end{theorem}

\begin{proof} Denote by $\cA_H$ ($\cC_{D^\la(H, E)}$, $\Mod^{\rm f.g.}_{\DHnn}$, $\Mod_E$, resp.) the category of admissible locally analytic $H$-representations on $E$-vector spaces (coadmssible $D^\la(H, E)$-modules, finitely generated $\DHnn$-modules, $E$-vector spaces, resp.). The functor in question is a composition of three functors 

\begin{numequation}\label{functors}
\cA_H \;\; \xrightarrow{V \rightsquigarrow V'_b} \;\; \cC_{D^\la(H, E)}  \;\; \xrightarrow{M \rightsquigarrow M_n} \;\; \Mod^{\rm f.g.}_{\DHnn}  \;\; \xrightarrow{X \rightsquigarrow X'} \;\; \Mod_E \;,
\end{numequation}

where $M_n = \DHnn\otimes_{D^\la(H, E)} M$. Objects in the category $\Mod^{\rm f.g.}_{\DHnn}$ carry a canonical $\DHnn$-module topology with respect to which they are $E$-vector spaces of compact type, by \ref{structure1} (ii) and \cite[5.1.1]{PSS}. 

\vskip8pt

By \cite[6.3]{S-T2} the first functor is an anti-equivalence of categories. A quasi-inverse functor is given by $M \rightsquigarrow M'_b$. We show that this functor is exact. Let  

$$0 \lra M_1 \xrightarrow{\iota} M_2 \xrightarrow{\pi} M_3 \lra 0$$

\vskip8pt

be an exact sequence of coadmissible modules. Consider the sequence of dual spaces

\begin{numequation}\label{exact1}
0 \lra M_3' \xrightarrow{\pi'} M_2' \xrightarrow{\iota'} M_1' \lra 0 \;.
\end{numequation}

$\pi'$ is clearly injective, and $\iota' \circ \pi'= 0$. If $\iota'(\lambda) = 0$, then $\lambda$ descends to a continuous linear form $\overline{\lambda}$ on $M_2/M_1$, equipped with the quotient topology. By the open mapping theorem, the continuous bijection $M_2/M_1 \ra M_3$ is a homeomorphism, and $\overline{\lambda}$ is a continuous linear form $\nu$ on $M_3$, and $\pi'(\nu) = \lambda$. Since any morphsim in $\cC_{D^\la(H, E)}$ is strict (proof of \cite[3.6]{S-T2}), so is $M_1 \xrightarrow{\iota} M_2$, and any continuous linear form on $M_1$ extends therefore to a continuous linear form on $M_2$ (cf. \cite[9.4]{NFA}). This shows that \ref{exact1} is exact.

\vskip8pt

Since $M \rightsquigarrow M'_b$ is an anti-equivalence of categories, any continuous $H$-morphism between admissible locally analytic representations is strict, and the maps in \ref{exact1} are thus strict too, i.e., \ref{exact1} is an exact sequence of topological vector spaces. 

\vskip8pt

The functor in the middle of \ref{functors} is exact by \ref{flatness} (ii). Now consider an exact sequence   

$$0 \lra M_1 \xrightarrow{\iota} M_2 \xrightarrow{\pi} M_3 \lra 0$$

\vskip8pt

of $\DHnn$-modules. By \ref{cptype}, the maps in this exact sequence are strict with closed image. It then follows from \cite[1.2]{S-T} that the sequence of dual spaces

$$0 \lra M_3' \xrightarrow{\pi'} M_2' \xrightarrow{\iota'} M_1' \lra 0 \;.$$

\vskip8pt

is exact too.
\end{proof}

\section{Example: a weak Fr\'echet-Stein structure on \texorpdfstring{$D^\la(\Zp, E)$}{}}
Let us consider the locally analytic functions $\ClaZp$. This is a locally convex $E$-vector space with a compact-type  topology and is naturally a $\Qp$-analytic representation over an $E$-vector space of $\Zp$. The dual $D^\la(\Zp,E)$ naturally inherits the structure of a Fr\'echet space. Let $X$ be the rigid analytic open unit disc centered at $0$ over $E$, then we can show for any $z\in X(\bbC_p)$ and $a\in \Zp$ the function $\kappa_z(a):= (1+z)^a$ defines a locally $\bbQ_p$-analytic character of $\Zp$. Given $\lambda\in D^\la(\Zp,E)$ we obtain a function $F_\lambda(z):= \lambda(\kappa_z)$ on $X(\bbC_p)$. $F_{\lambda}(z)$ is a rigid analytic function on $X$ and can thus be expanded as a power series $\sum_{n=0}^{\infty}a_nz^n$ where $a_n\in E$ and $\lim_{n\ra \infty}|a_n|r^n=0$ for all $r<1$. If we denote the $E$-valued rigid analytic functions on $X$ by $\cO(X):=\Big\{\sum_{n=0}^{\infty} a_n T^n \,\Big| a_n\in E \;\text{s.t}\; \lim_{n\ra \infty} a_nr^n\ra 0\; \forall\; r < 1 \Big\}$ then it is known \cite[Thm. 1.3, Section. 2.3.4]{AM} that 
\begin{align*}
    D^\la(\Zp,E)&\ra \cO(X)\\
    \lambda&\ra F_\lambda
    \end{align*}
    is an isomorphism. 
    
\vskip8pt    
Let us consider, $C^\la(\Zp)_{p^n\cm-\an}$, this is the space of locally analytic functions such that around any point on $\Zp$ the function is rigid analytic on an open unit disk of radius $|p|^{n}$, i.e., 

$$\ClaZp_{p^n\cm-\an}= \bigoplus_{a\in \Z/p^{n+1}\Z}\cO(a+p^n\cm)$$ 

\vskip8pt

where $\cO(a+p^n\cm):=\Big\{\sum_n c_n(T-a)^n \,\Big|\; c_n\in E, \lim_{n\ra \infty}c_nr^n=0\;\forall\;r< |p|^n \Big\}$. By compactness of $\Zp$ we see that 
$$C^\la(\Zp, E)= \varinjlim_{n\geq 0}\Big(\bigoplus_{a\in \Z/p^{n+1}\Z}\cO(a+p^n\cm)\Big)\;,$$

\vskip8pt

and dually (cf. \cite[16.10 (iii)]{NFA})

$$D^\la(\Zp, E) =\varprojlim_{n\geq 0}\Big(\bigoplus_{a\in \Z/p^{n+1}\Z}\cO(a+p^n\cm)\Big)'_b \;.$$ Our main objective is to analyze, $\bigoplus_{a\in \Z/p^{n+1}\Z}\cO\Big(a+p^n\cm\Big)$.

\begin{lemma}\label{estimate}
Let $b\in\bbZ$ be a representative of an element of $\Z/p^{n+1}\Z$. Let $a=b+p^{n+1}c$ with $|c|=1$ and $\frac{b-i}{p^{n+1}}\neq -c\, (\mbox{\,\rm mod\;} p)\;,0\leq i <p^{n+1}$, then

$$\max_{b \in \Z/p^{n+1} \Z}  \Big|{a\choose k}\Big|= \Big|\lfloor\frac{k}{p^{n+1}}\rfloor!\Big|^{-1}\;.$$
\end{lemma}

\begin{proof} We have,
\begin{numequation}\label{absvalue}
|a-i|= \begin{cases} |b-i| \hskip50pt i\neq  b \, (\mbox{\,\rm mod\;}p^{n+1}\Zp)\\
|p^{n+1}|\hskip50pt i= b \, (\mbox{\,\rm mod\;}p^{n+1}\Zp)
\end{cases}
\end{numequation}

Let $\alpha_{b,j}:=\#\{0\leq i < k\,|\; |b-i|\leq |p|^j\}\, ,$
\vskip8pt
note with our choices of $a$ and $b$, 
\begin{numequation}\label{alpha}
\alpha_{b, n+1}= \{0\leq i <k \,|\; |b-i|= |p|^{n+1}\; .\} 
\end{numequation}

Using (\ref{absvalue}) and (\ref{alpha}) we get that for a fixed $b$,
\begin{align*}
    &|a(a-1)\ldots(a-k+1)|= \prod_{\substack{i= b(\, (\mbox{\,\rm mod\;} p^{n+1})\\0\leq i <k}}|p|^{n+1}\prod_{\substack{i\neq b(\, (\mbox{\,\rm mod\;} p^{n+1})\\0\leq i<k}}|a-i| \\
    &= |p|^{(n+1)\alpha_{b,n+1}}|p|^{n(\alpha_{b,n}-\alpha_{b,n+1})}|p|^{(n-1)(\alpha_{b,n-1}-\alpha_{b,n})}\ldots|p|^{(\alpha_{b,1}-\alpha_{b,2})}\\
    &=|p|^{\sum_{j=0}^{n+1}\alpha_{b,j}}
\end{align*}
\vskip8pt
We note that $\alpha_{b,j}=\lfloor\frac{k}{p^j}\rfloor$ or $\lfloor\frac{k}{p^j}\rfloor+1$ and if we write $k= m.p^{n+1}+ b_0$, then $\alpha_{b_0, j}= \lfloor\frac{k}{p^j}\rfloor \;\forall\; 1\leq j\leq n+1$.

\vskip8pt

So we obtain, 

$$\max_{b\in \Z/p^{n+1}\Z} |a(a-1)\ldots(a-k+1)|= |p|^{\sum_{j=1}^{n+1} \lfloor\frac{k}{p^j}\rfloor}= \Big|\frac{k!}{\lfloor \frac{k}{p^{n+1}}\rfloor!}\Big|$$

\vskip8pt
Or equivalently, 
$$\max_{b\in \Z/p^{n+1}\Z} \Big|{a\choose k}\Big|= \Big|\Big\lfloor\frac{k}{p^{n+1}}\Big\rfloor!\Big|^{-1}$$
\end{proof}

\begin{theorem} Let us define $r_n:= |p|^{\frac{1}{(p-1)p^n}}$ and $$\cF_n:= \Big\{\sum_{k=0}^{\infty} c_k {x \choose k}\;\Big|\lim_{k\ra \infty}c_kr_n^k\epsilon^k\ra 0 \;\mbox{for all}\, \epsilon < 1 \Big\}$$ Given $f\in  \bigoplus_{a\in \Z/p^{n+1}\Z}\cO(a+p^n\cm)\,,$ the map $ev_f:\Zp\ra E$ defined by restricting $f$ to $\Zp$ is locally analytic and its Mahler series lies in $\cF_n$. Conversely, every element of  $\cF_n$ is obtained as the Mahler series of some $g\in  \bigoplus_{a\in \Z/p^{n+1}\Z}\cO(a+p^n\cm)$. The resulting map $\bigoplus_{a\in \Z/p^{n+1}\Z}\cO(a+p^n\cm)\ra \cF_n$ is a topological isomorphism.
\end{theorem}

\begin{proof}
The crux of the proof lies in estimating $p$-adic norms of ${x \choose k}$ for $x\in a+p^n\cm $. For such an $x$ we write $x= a+ p^ny$, maximum-modulus principle allowing us to choose $|p|<|y|<1$.

\vskip8pt

If we choose representatives of $\Zp/p^{n+1}\Zp$ in $\Zp$ as $a= b+ p^{n+1} c$ with $|c|= 1$ 
then we have,
\vskip8pt

\begin{numequation}\label{absvalue2}
|x-i|= \begin{cases} |a-i| \hskip50pt i\notin a+p^{n+1}{\cO_{\bbC_p}}\\
|p^ny|\hskip54pt i\in a+p^{n+1}\cO_{\bbC_p}p
\end{cases}
\end{numequation}
 Using (\ref{absvalue2}) we get,
    \begin{align*}
    &|x(x-1)\ldots(x-k+1)| = |(a+p^ny)(a-1+p^ny)\ldots(a-k+1+p^ny)|\\
    &= \prod_{\substack{a-i\in p^{n+1}\cO_{\bbC_p}\\ 0\leq i < k}} |p^ny| \prod_{\substack{a-i\notin p^{n+1}\cO_{\bbC_p}\\ 0\leq i < k}} |a-i| = \prod_{0\leq i < k} |a-i|\prod_{\substack{\nu_p(a-i)= n+1\\ 0\leq i < k}}|\frac{y}{p}|\\
    &=\Big|{a \choose k}k!\Big|\prod_{\substack{\nu_p(a-i)= n+1\\ 0\leq i < k}}|\frac{y}{p}|
\end{align*}
By Thm.\ref{estimate}, 
$$\max\Big\{\Big|{a \choose k}\Big|_{\bbB_{\leq p^{n+1}}(b)}\,\Big|\; b\in \Zp/p^{n+1}\Zp\Big\} = \Big|\lfloor\frac{k}{p^{n+1}}\rfloor!\Big|^{-1}$$

\vskip8pt where $\bbB_{\leq p^{n+1}}(b)$ is the closed disk of radius $p^{n+1}$ around $b$. And $\#\{\nu_p(a-i)= n+1\}\leq \lceil \frac{k-b}{p^{n+1}}-1\rceil$ with the bound being sharp. Combining all of these together we get for $x\in a+ p^n\cm$
\begin{align*}
    \Big|{x \choose k}\Big|&=|\lfloor\frac{k}{p^{n+1}}\rfloor!|^{-1}|\frac{y}{p}|^{\lceil \frac{k-b}{p^{n+1}}-1\rceil} \text{\;with\, equality\, for\, certain $a$ and $b$},\\&\leq |p|^{-\frac{k}{(p-1)p^{n+1}}-\frac{k}{p^{n+1}}}|y|^{\frac{k}{p^{n+1}}}= r_n^{-k}(|y|^{\frac{1}{p^{n+1}}})^{k}\text{\;(asymptotically)}\\&\leq r_n^{-k}\epsilon^k\;\; \forall \;\epsilon< 1.   
\end{align*}
 The second line follows from the facts $|n!|\sim |p|^{\frac{n}{p-1}}$ and $\lceil \frac{k-b}{p^{n+1}}-1\rceil\sim \frac{k}{p^{n+1}}$ asymptotically.

Let $f \in \bigoplus_{a\in \Zp/p^{n+1}\Zp}\cO(a+p^n\cm)$ (considered as an element of  $C^\la(\Zp, E)$) be written as $f=\sum_k c_k {x \choose k} $ (Mahler series expansion of $f$). Convergence of this expression for $x\in a+p^n\cm$ gives $\lim_{k\ra \infty} c_kr_n^{-k}\epsilon^k\ra 0$ for every $\epsilon < 1$.

\vskip8pt That there are no $ 0< r < r_n$ such that  $\lim_{k\ra \infty}c_k r^k\epsilon^k\ra 0$ can be deduced from the logarithmic (in $k$) growth of
$\frac{|\lfloor\frac{k}{p^{n+1}}\rfloor!|^{-1}|\frac{y}{p}|^{\lceil \frac{k-b}{p^{n+1}}-1\rceil}}{|p|^{-\frac{k}{(p-1)p^{n+1}}-\frac{k}{p^{n+1}}}|y|^{\frac{k}{p^{n+1}}}}$
\end{proof}
Now we prove the following theorem
\begin{theorem}
Let  a space of sequences $(c_k)_k, c_k\in E$ be defined as 
$$\cF_n:= \Big\{\sum_{k= 0}^{\infty} c_k{x \choose k}\, \Big|\lim_{k\ra \infty} c_kr_n^{-k}\epsilon^k\ra 0\,  \mbox{for all}\, \epsilon< 1\Big\}\,.$$

\vskip8pt

Let us define for a fixed $\epsilon<1$ the normed space 
$$\cF_{n,\epsilon}:=\Big\{\sum_{k=0}^{\infty} c_k {x\choose k}\;\mbox{where\;} \,\Big|\sum_k  c_k {x \choose k}\Big|_\epsilon:= \sup\{|c_k|r_n^{-k}\epsilon^k\}_k \Big\}\,.$$

\vskip8pt

We have $\cF_n= \varprojlim_{\epsilon < 1}\cF_{n,\epsilon}$. Then the continuous dual of this space is  the space of sequence $(d_k)_k$, formally written as $\sum d_kT^k$ can be explicitly described as 
$$\cG_n:=\Big\{\sum_{k=0}^{\infty} d_kT^k\,\Big| \;\exists \;R> r_n\; s.t. \lim_{k\ra \infty} d_kR^k\ra 0 \Big\}\,.$$
\end{theorem}

\begin{proof}
Let $\phi$ be a continuous linear functional from $\cF_n\ra E$. By continuity, there exists $\epsilon < 1$ and $C> 0$ such that for each $f(x):= \sum_{k=0}^{\infty} c_k {x \choose k}$ we have $|\phi(f(x))|\leq C |f(x)|_{\epsilon}$. Now if we write $\phi({x \choose k}) = d_k$, then we get 
\begin{align*}
    &|d_k| \leq C r_n^{-k}\epsilon^k\\
   \implies &|d_k|r_n^k(\frac{1}{\epsilon})^k\leq C\\
   \implies &\lim_{k\ra \infty}|d_k|R^k = 0 \;\forall\; R < \frac{r_n}{\epsilon}
    \end{align*}
\end{proof}
As a consequence of the previous theorem we get a neat description of the dual of $\bigoplus_{a\in \bbZ/p^{n+1}\bbC_p}\cO(a+p^n\cm)$ 
\begin{cor}\label{pnan}
$(\bigoplus_{a\in \bbZ/p^{n+1}\bbC_p}\cO(a+p^n\cm))'_b\cong \cO_{r_n}(X)^{\dagger}$, where $\cO_{r_n}(X)^{\dagger}$ is the space of overconvergent functions on a closed disk of radius $r_n$
\end{cor}
\begin{proof}
On the level of sets this immediately follows from the above theorem once we explicitly recall the description of 

$$\cO_{r_n}(X)^{\dagger} := \Big\{\sum_{n=0}^\infty a_n T^n \; \Big| \; \lim_{n \ra \infty} a_n R^n\ra 0 \mbox{ for some } R  > r_n \Big\} \;.$$

\vskip8pt 

In fact it can be shown that the isomorphism is indeed an isomorphism of topological vector spaces.
\end{proof}
\bibliographystyle{abbrv}
\bibliography{mybib}

\end{document}